\documentclass[12pt]{amsart}
\usepackage{amsfonts}
\usepackage{amssymb,amsmath,amscd,cancel}
\usepackage{epsfig}
\usepackage{graphicx}
\usepackage[all]{xy}

\newtheorem{thm}{\sc Theorem}[section]

\newtheorem{lem}[thm]{\sc Lemma}
\newtheorem{cor}[thm]{\sc Corollary}
\theoremstyle{definition}

\theoremstyle{definition}

\theoremstyle{definition}

\theoremstyle{definition}

\numberwithin{equation}{section}

\input{tcilatex}

\begin{document}
\title[The kinematic formula]{The kinematic formula in the 3D-Heisenberg group} 
\author[Y.-C. Huang]{Yen-Chang Huang}
\address{Department of Mathematics, Xiamen University, Malaysia}
\email{ychuang@xmu.edu.my}

\begin{abstract} By studying the group of rigid motions, $PSH(1)$, in the 3D-Heisenberg group $H_1$, we define the density and the measure for the sets of horizontal lines. We show that the volume of a convex domain $D\subset H_1$ is equal to the integral of length of chord over all horizontal lines intersecting $D$. As the classical result in integral geometry, we also define the kinematic density for $PSH(1)$ and show the probability of randomly throwing a vector $v$ interesting the convex domain $D\subset D_0$ under the condition that $v$ is contained in $D_0$. Both results show the relationship connecting the geometric probability and the natural geometric quantity in \cite{CHMY1} approached by the variational method.
\end{abstract}

\maketitle

\section{Introduction}
We adapt the methods of integral geometry to study the 3D-Heisenberg group which is a non-compact CR manifold with zero Tanaka-Webster torsion and zero Tanaka-Webster curvature. After the brief history of integral geometry and the introduction of CR geometry, we will illustrate how our formulas show the relationship connecting these fields.

The roots of integral geometry may date back from geometric probability and the study of invariant measures by integration techniques, which consider the probability of random geometric objects interacting to each other under a group of transformations, for example, Buffon's needle problem and Bertrand's paradox. In the late nineteenth and early twentieth century, a variety of problems in geometric probability arose and lead to systematical studies in this field. Works of Crofton, Poincar\'e, Sylvester and others build up the foundation of integral geometry. A series of articles related to the developments of geometric probability in this age was elaborated by Maran \cite{Moran1, Moran2}, Little \cite{Little}, and Baddeley \cite{Baddeley}. When the concept of invariant measure had became clear, Wilhelm Blaschke \cite{Blaschke} and his school initiated integral geometry. Santal\'o's book \cite{San} has been one of the most important books on the subject; Howard's book \cite{Howard} deals with the cases in Riemannian geometry; Zhou \cite{Zhou2, ZL} derives several integral formulas for submanifolds in Riemannian homogeneous spaces; The book of Schneider and Weil \cite{SW} includes the fundamental knowledge of integral geometry and recent development of integral and stochastic geometry. Due to page restriction, we refer the delicate surveys \cite{Zhang} \cite{Zhou1}.

CR geometry studies the geometry of the boundary of a smooth strictly pseudo-convex domain in $\mathbb{C}^n$, and then develops to the abstract CR manifolds. The foundational work was built by Chern-Moser \cite{CM} in 1974, and a closely connected work is given by Webster \cite{Webster} and Tanaka \cite{Tanaka} independently; they introduced the pseudohermitian geometry where a connection is given associated to a choice of a contact form and curvature invariants. We point out that the relationship between CR geometry and pseudohermitian geometry have the strong analogy to that between conformal geometry and Riemannian geometry \cite{BFG}. The introductory surveys, emphasizing recent development of three dimensional pseudohermitian geometry, can be found in \cite{Yang2, Yang1}.

Next we give the background of our studying target. The 3D-Heisenberg group $H_1$ is the Euclidean space $\mathbb{R}^3$ with a group multiplication (see Section \ref{Preliminary}). A rigid motion (called a \textit{pseudohermitian transformation}) in $H_1$ is a diffeomorphism $\Phi$ defined on $H_1$ preserving the CR structure $J$ and the contact form $\Theta$, namely,
\begin{align*}
\Phi_{*} J&=J \Phi_* \text{ on the contact plane } \xi,  \\  \Phi^{*} \Theta &= \Theta \text{ in } H_1.
\end{align*}
Denote the group of pseudohermitians by $PSH(1)$. Similar to the usual group of translations in $\mathbb{R}^3$, in \cite{CHL} we show that $PSH(1)$ consists of left-translations $L_Q$ and rotations $R_\alpha\in SO(2)$, (as shown (\ref{multi})). The understanding of the structure of $PSH(1)$ plays the important role when using the method of moving frames (see equations (\ref{multi})(\ref{changecoo}). Let $\Sigma$ be an embedded surface in $H_1$ and denote $S_\Sigma$ the set of points on $\Sigma$ where the contact plane $\xi$ is tangent to the tangent plane $T\Sigma$ of the surface. Given an orthonormal basis $(e_1, e_2, T)$ defined on $\Sigma\setminus S_\Sigma$ w.r.t. the Levi metric, where $e_1\in T\Sigma\cap \xi$ on $S_\Sigma$ and $e_2=Je_1$, and $T$ is the Reeb vector field. Suppose $(\omega^1, \omega^2, \Theta)$ is the dual basis of $(e_1, e_2, T)$. To our purpose, two geometric quantities naturally following the geometry in $H_1$ should be introduced
\begin{align}
V(D)&=\frac{1}{2}\int_D \Theta\wedge d\Theta, \label{volume2}\\
\text{p-Area}(\Sigma)&=\int_\Sigma \Theta\wedge w^1, \label{parea}
\end{align} which are called, respectively, the \textit{volume} of the domain $D\subset H_1$ and the \textit{p-Area} of its boundary $\partial D=\Sigma$. The authors in \cite{CHMY1} show that the volume actually is the Lebesque volume in $\mathbb{R}^3$, while the p-area form $\Theta \wedge \omega^1$, resulted from the variation of volume $V$ along the Legendrian normal $e_2$, is globally defined on $\Sigma$ and vanishes on $S_\Sigma$. The similar notions of volume and area are also studied by \cite{RR1,RR2, FSC,CDG} when considering the $C^1$-surface $\Sigma$ enclosing a bounded set $D$; the area of $\Sigma$ coincides with the $H_1$-perimeter of of $D$. The authors adapt the usual normal vector perpendicular to the surface, while we consider the Legendrian normal which later on develops toward to the studies of umbilic hypersurfaces in higher dimensional Heisenberg Group $H_n$, see \cite{CCHY1}.

A \textit{horizontal line} is a straight line in $\mathbb{R}^3$ such that its velocity vector is always tangent to the contact plane. In the next section, we shall show that every horizontal line $G$ can be uniquely represented, up to an orientation, by three parameters $(p,\theta, t)$ (equivalently, by the base point $B\in G$). Denote by $G_{p,\theta,t}$ and therefore the three-form $dp\wedge d\theta\wedge dt$ is invariant under $PSH(1)$ (see (\ref{density})). Follow the classical integral geometry \cite{San}\cite{Chern}, we give the analogous definition of measure for a set of horizontal lines.

\begin{definition}\label{densityoflines}
The \textit{measure} of a set $X$ of horizontal lines $G_{p,\theta,t}$ is defined by the integral over the set $X$,
$$m(X)=\int_X dG,$$
where the diffeerential form $dG=dp\wedge d\theta \wedge dt$ is called the \textit{density} for sets of horizontal lines.
\end{definition}

Remark that up to a constant factor, this density is the only one that is invariant under motions of $PSH(1)$. In convention, the density will always be taken at absolute value.

Recall that the Crofton's formula: for a convex domain in $M^2\subset \mathbb{R}^2$ with the length $\ell$ of boundary $\partial M^2$, we have (\cite{San}, Chapter 4.2)
\begin{align*}
\int_{\forall G,\ G\cap D\neq \emptyset}dG= 2 \ell.
\end{align*}
By using the method of moving frame for the convex domain $D\subset H_1$ with boundary $\partial D=\Sigma$, we show the Crofton-type formula (\cite{CHL} Theorem 1.11)
\begin{align}\label{integraldg}
\int_{\forall G,\ G\cap D\neq \emptyset}dG= 2\cdot \text{p-Area}(\Sigma).
\end{align} Also, in the Euclidean plane (\cite{San}, Chapter 3), it shows that the area of a convex domain $D\subset \mathbb{R}^2$ is equal to the integral of length of chord $\sigma$ over all lines intersecting $D$,
\begin{align*}
\int_{\forall G,\ G\cap D\neq\emptyset} \sigma dG= \pi\cdot  \text{area}(D),
\end{align*} where $G$ stands for the line in $\mathbb{R}^2$, $dG=dp\wedge d\theta$ is the density of set of lines. Our first main result is the analogy in $H_1$:

\begin{theorem}\label{volume1}
Let $D$ be a convex domain with boundary $\Sigma:= \partial D$ in $H_1$. Suppose $\gamma(s)$ is an oriented horizontal line $G_{p,\theta, t}$ intersecting with $D$ by the length of chord $\sigma$ w.r.t. the Levi metric. Then
$$\int_{\forall G, \ G\cap D\neq \emptyset}\sigma dG = 2\pi V(D),$$ where $dG= dp\wedge d\theta \wedge dt$ is the density of the horizontal lines and $V(D)$ is the Lebesque volume of $D$ in $\mathbb{R}^3$.
\end{theorem}

Once we have defined the measure of sets of lines, the \textit{probability} (\cite{San}, Chapter 2) that a random line is in the set $X$ when it is known to be in the bigger set $Y$ containing $X$ can be defined by the quotient of the measures
\begin{align}\label{conditional}
P(G\in X|G\in Y)=\frac{m(G\cap X\neq\emptyset)}{m(G\cap Y\neq\emptyset)}.
\end{align}
Thsu, by the conditional probability $(\ref{conditional})$ and (\ref{integraldg}), we immediately have the corollaries:

\begin{cor}
Given a convex 3-domain $D$ in $H_1$ with boundary $\partial D=\Sigma$, and randomly throw an oriented horizontal line $G$ interesting $D$ once a time with chord length $\sigma$. The average chord length of the lines interesting $D$ is
$$\frac{m(\sigma; G\cap D \neq \emptyset)}{m(G;G\cap D\neq \emptyset)}=\frac{\int_{G\cap D\neq \emptyset}\sigma dG}{\int_{G\cap D\neq \emptyset}dG}=\frac{2\pi V(D)}{2\cdot \text{p-Area}(\Sigma)}.$$
\end{cor}

Because of the Lie group structure of $H_1$, we can define the the moving frame $(Q; e_1(Q),e_2(Q),T)$ at the point $Q$ by moving the standard frame $(O;\mathring{e}_1(O), \mathring{e}_2(O), T)$ (defined in (\ref{standard})) at the original $O$ along the horizontal line $G$ under translations in $PSH(1)$ (see Section \ref{Preliminary}). Note that $span\{e_1(Q), e_2(Q)\}=\xi_Q$ for any $Q\in H_1$. There exists an one-to-one correspondence associated to $PSH(1)$ and the moving frames, and thus any element in $PSH(1)$ is uniquely parametrized by four variables $(a,b,c,\phi)$, where the point $Q=(a,b,c)$ dominates the position of the frame, $\phi$ the rotation on $\xi_Q$. Moreove, we shall show the following four-form is invariant under $PSH(1)$, and hence one defines it as the kinematic density of $PSH(1)$.
\begin{definition}The invariant $4$-form
\begin{align}\label{kdensity}
dK:=da\wedge db\wedge dc\wedge d\phi 
\end{align} is called the \textit{kinematic density}    for the group of motions $PSH(1)$ in $H_1$,
where $\phi$ is the angle from the standard vector $\mathring{e}_1(Q)$ to the frame vector $e_1(Q)$, $Q=(a,b,c)$.
\end{definition} 
We shall also show that the invariant volume element of $PSH(1)$ has the other expression
\begin{align}\label{kdensity2}
dK= dG\wedge dh,
\end{align}where $h$ is the oriented distance from the base point $B\in G$ to $Q$ w.r.t. the Levi-metric, $dG$ the density in Definition \ref{densityoflines}.

By integrating $dK$ over a domain in $PSH(1)$, one gets the measure $m$ of the corresponding set of motions (as \cite{San}, called the \textit{kinematic measure}). This kinematic measure is actually somewhat a unexpected geometric quantity, p-Area, so it is worth to believe that the suitable approach can help us understand CR (pseudohermitian) geometry from the viewpoint of integral geometry. The following result gives an evidence of the connections between these two fields.

\begin{theorem}\label{randomm}
Given a convex 3-domain $D$ in $H_1$ with boundary $\partial D=\Sigma$. Randomly throw an oriented line segment $v$ (may not be necessarily horizontal) with length $\ell$ w.r.t. the Levi metric (Fingure \ref{Fig3}). The measure of random line segments intersecting the domain $D$ is
\begin{align}\label{throwing}
m(v; v\cap D\neq \emptyset):=\int_{v\cap D\neq \emptyset}dK=2\pi V(D)+2 \ell  \cdot \text{p-area}(\Sigma),
\end{align}
where $dK$ is the kinematic density, $V(D)$ the Lebseque volume of $D$, and p-Area($\Sigma$) the p-area of the surface $\Sigma$.
\end{theorem}
\begin{figure}[!ht]
   \includegraphics[width=0.8 \textwidth]{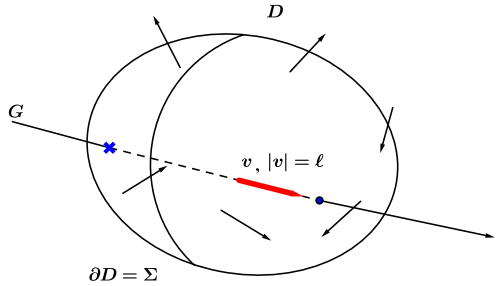}
   \centering \caption{Random intersections of vector $v$}
    \label{Fig3}
\end{figure}
\begin{remark}
We take for granted that $v\cap D\neq\emptyset$ in the sense of $v\cap D\neq   \emptyset$ or $v\cap \Sigma \neq\emptyset$. When restricted our attention to $v\subset D$ only, the equation $(\ref{throwing})$ becomes
$$\int_{G\cap D\neq \emptyset, v\subset D} (\sigma-\ell) dG$$
and the measure of $v\subset D$ is
$$m(v;v\subset D)=2\pi V(D)-2\ell \cdot \text{p-Area}(\Sigma).$$
\end{remark}

With the measure, we immediately have the probability of containment problem.
\begin{cor}
Let $D_i$ be convex 3-domains with boundary $\partial D_i=\Sigma_i$ for $i=1,2$, such that $D_1\subset D_2$. The probability of randomly throwing a vector $v$ with length $\ell$ w.r.t. Levi metric in $D_2$ intersecting $D_1$ is
$$P(v\cap D_1|v\ \cap D_2)=\frac{2\pi V(D_1)+2\ell\cdot \mbox{ p-Area}(\Sigma_1)}{2\pi V(D_2)+2\ell\cdot\mbox{ p-Area}(\Sigma_2)}.$$
\end{cor}

\textbf{Acknowledgment.} Part of the contents in this article was completed during the author's postdoctoral at National Central University, Taiwan. I would like to thank my mentor, Prof. Hung-Lin Chiu, for his introduction and inspiration to this problem.

\section{Preliminary}\label{Preliminary}

For basic knowledge of the Heisenberg group, we refer \cite{CHMY1}(appendix), \cite{RR1}, and our previous work \cite{CHL}; also, \cite{CHMY2, Yang1, Yang2}. There are some works from sub-Riemannian approaches, e.g. \cite{CCG, RR1, CDPT, LM, RR2, CDG, FSC}.
The three dimensional Heisenberg group $H_1$ is $\mathbb{R}^3$, as a set, together with the group multiplication (left-invariant translation)
$$L_{(a,b,\tilde{t})}\circ (x,y,t) = (a+x,b+y, c+\tilde{t}+bx-ay).$$
$H^1$ is a 3-dimensional Lie group. Any left invariant vector field, at point $(x,y,t)$, is a linear combination of the following standard vector fields:
\begin{align}\label{standard}
\mathring{e}_1&=\frac{\partial}{\partial x}+y\frac{\partial}{\partial t}, \ \ \ \mathring{e}_2=\frac{\partial}{\partial y}-x\frac{\partial}{\partial t},\\
T&=\frac{\partial}{\partial t}.
\end{align}
The standard contact structure $\xi=span\{\mathring{e}_1, \mathring{e}_2$\} on $H_1$ is the subbundle of the tangent bundle $TH$; equivalently, we can also define $\xi$ to be the kernal of the standard contact form
$$d\Theta=dt+xdy-ydx.$$
The standard CR structure on $H_1$ is the almost complex structure $J$ defined on $\xi$ such that
\begin{align*}
J^2=- I, J(\mathring{e}_1)=\mathring{e}_2,\  J(\mathring{e}_2)=-\mathring{e}_1.
\end{align*} On $\xi$ a natural metric $$L_\Theta(X,Y)=\frac{1}{2}d\Theta(X, J Y)=\frac{1}{2}\left((dx)^2+(dy)^2\right),$$ called the \textit{Levi metric}, ; for $v\in \xi$ we define the length of $v$ by $|X|=\sqrt{<X,X>}$.

In our previous work \cite{CHL}, we show that any pseudo-hermitian transformation $\Phi_{Q,\alpha}$ can be represented by a left-invariant translation $L_Q$ for $Q=(a,b,c)\in \mathbb{R}^3$ and a rotation $R_\alpha\in SO(2)$. There exits the one-to-one correspondence between the group actions and the matrix multiplication
\begin{equation}\label{multi}
\Phi_{Q,\alpha}(x,y,t):=L_{(a,b,c)}\circ R_\alpha (x,y,t) \longleftrightarrow
\left(
\begin{array}{cccc}
1&0&0&0\\
a&\cos\alpha & -\sin\alpha&0\\
b&\sin\alpha &\cos\alpha &0\\
c&b&-a&1
\end{array}
\right)
\left(
\begin{array}{c}
1\\
x \\
y \\
t
\end{array}
\right).
\end{equation}%
Let $\Sigma$ be a smooth hypersurface in $H_1$. Recall that a point $p\in \Sigma$ is called \textit{singular} if the contact plane conincides with the tangent plane at $p$, namely, $\xi_p=T_p\Sigma$; otherwise $p$ is called \textit{regular}. Denote $S_\Sigma$ is the set of singular points. It is easy to see that $S_\Sigma$ is a closed set.  At each regular point $p$, there exists (unique up to a sign) a vector $e_1\in \xi_p\cap T_p\Sigma$, which defines a one dimensional foliation consisting of integral curves of $e_1$ and we call it the \textit{characteristic curves}. The vector $e_2=Je_1$ perpendicular to $e_1$, in the sense of Levi-metric, is called \textit{Legendrian normal} or \textit{Gauss map} \cite{CHMY1}. Let $D\subset H_1$ be a smooth domain with boundary $\partial D=\Sigma$ and $(\omega^1,\omega^2, \Theta)$ be the dual basis of $(e_1,e_2, T)$ where $T$ is the Reeb vector field such that $\Theta(T)=1$ and $d\Theta(\cdot, T)=0$. Cheng-Hwang-Malchiodi-Yang \cite{CHMY1} study the minimal surface in $H_1$ via the variational approach and define the volume and the p-Area respectively
\begin{align}
V(D)&=\frac{1}{2}\int_D \Theta\wedge d\Theta, \label{volume3}\\
\text{p-Area}(\Sigma)&=\int_\Sigma \Theta\wedge w^1. \label{parea}
\end{align} We point out that $\frac{1}{2}$ is a normalization constant and this volume is just the usual Euclidean volume. While p-Area comes from a variation of the surface $\Sigma$ in the normal direction $fe_2$ for some suitable function with compact support on the regular points of $\Sigma$.
Note that we can continuously extend $\Theta\wedge \omega^1$ over the singular set $S_\Sigma$ such that it vanishes on $S_\Sigma$. Thus, p-Area is globally defined on $\Sigma$. In the same paper, the authors also define the p-mean curvature and of associated p-minimal surfaces; Malchiodi \cite{Malchiodi} summarize the other two equivalent definitions for the p-mean curvature.

\section{Invariants for sets of horizontal lines}
In the section, we will derive the invariants for the set of horizontal lines as shown in Definition \ref{density} and the other expression of the kinematic density for $PSH(1)$.

Given a regular curve $\gamma: t\in I \mapsto H_1$. Its velocity can always be decomposed into the part tangent to the contact plane $\xi$ and the other orthogonal to $\xi$ w.r.t the Levi metric, namely,
\begin{align*}
\gamma'(t)= \underbrace{\gamma'_\xi (t)}_{\in \xi}+\underbrace{\gamma'_T(t)}_{\in T}.
\end{align*}
A \textit{horizontally regular curve} is a regular curve with non-zero contact part
$$\gamma_{\xi}'(t)\neq 0 \text{ for all } t\in I.$$
In \cite{CHL} (Proposition 4.1) the authors show that any horizontally regular curve can be parametrized by the horizontal arc-length $s$ such that $|\gamma'_\xi(s)|=1$. Throughout the article, we always assume that the curve (or line) is parametrized under this condition. Moreover, if $\gamma(s)$ is a curve joining points $A=\gamma(s_0)$ and $B=\gamma(s_1)$, we have the distance
\begin{align}\label{distance}
|AB|=\int_{s_0}^{s_1}|\gamma'_\xi(s)|ds.
\end{align}

Now we characterize the horizontal lines. A \textit{horizontal line} $G$ in $H_1$ is uniquely determined by three parameters $(p,\theta, t)$: the angle $\theta$ $(0\leq \theta <  2\pi)$ of its projection $\pi(G)$ onto $xy$-plane from the positive $x$-axis to the line perpendicular to $G$, the distance $p\geq 0$ from the origin to the \textit{footpoint} $b: (p\cos\theta,p\sin\theta,0)\in \pi(G)$, and the hight of \textit{base point} $B$ by lifting of $b$ to $G$. (as Figure \ref{proj}.)
\begin{figure}[ht]
    \centering
    \includegraphics[scale=0.5]{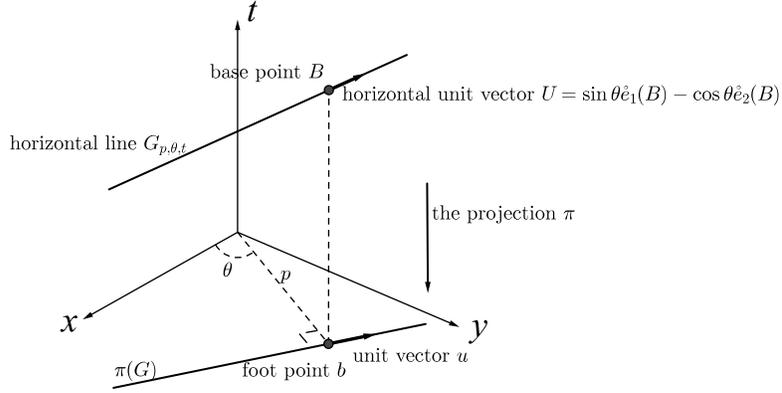}
    \caption{horizontal line $G_{p,\theta,t}$}
    \label{proj}
\end{figure}
Use (\ref{standard}), the curve $\gamma(s)$ can be parametrized by the base point $B$ and the unit vector $U$
\begin{align}
\gamma(s)&=B+s U \nonumber \\
&=(p\cos\theta, p\sin\theta, t)+s(\sin\theta \mathring{e}_1(B)-\cos\theta\mathring{e}_2(B)) \nonumber \\
&=(p\cos\theta+s\cdot \sin\theta, p\sin\theta-s\cdot \cos\theta, t+s(y\sin\theta+x\cos\theta)). \label{pcurve}
\end{align}
(\ref{pcurve}) implies that the points $(x,y,z)$ on $\gamma$ satisfy the conditions
\begin{align*}
p&=x\cos\theta + y \sin\theta, \\
z&=(x\sin\theta - y\cos\theta)p + t.
\end{align*}
Therefore, we have the following expression for any horizontal line
\begin{align}G_{p,\theta,t}=
\Big\{
\begin{array}{rll}
(x,y,z)\in \mathbb{R}^3  &\bigg|&
\begin{array}{l}
p=x\cos\theta + y \sin\theta,\\
z=(x\sin\theta - y\cos\theta)p + t
\end{array}
\end{array}\Big\}.
\end{align}
Suppose the horizonal line $G'_{p',\theta',t'}$ is obtained by $G_{p,\theta,t}$ transformed under a pseudohermition transformation $L_Q\circ \Phi_\alpha$. By the matrix multiplication (\ref{multi}), it is easy to calculate the transformed line $G'_{p',\theta',t'}$ satisfying
\begin{align}\label{changecoo}G'_{p',\theta',t'}:
\left\{
\begin{array}{ll}
\theta'&=\theta-\alpha\\
p'&=p+a\cos(\theta-\alpha)+b\sin(\theta-\alpha)\\
t'&=2p(a\sin(\theta-\alpha)+b\cos(\theta-\alpha))+c+t,
\end{array}
\right.
\end{align} and therefore
\begin{align}\label{density}
dp\wedge d\theta \wedge dt = dp'\wedge d\theta'\wedge dt'.
\end{align} To the end, the 3-form $dp\wedge d\theta \wedge dt$ is invariant under $PSH(1)$.
Moreover, if the measure of a set $X$ of horizontal lines is defined by any integral of the form
\begin{align}
m(X)=\int_X f(p,\theta,t)dp\wedge d\theta \wedge dt
\end{align} for some function $f$. Follow the spirit of classical integral geometry and geometric probability, the most natural measure is to be \textit{invariant under the group of rigid motion} $PSH(1)$ in $H_1$. If we want the measure $m(X)$ to equal the measure of transformed set $m(X')=m(L_Q\circ\Phi_\alpha X)$ for any set $X$ and any motion, by $(\ref{changecoo})(\ref{density})$, the function $f$ has to be constant. Choose the constant to be unity one gets the Definition $\ref{densityoflines}$.

Next we derive the alternative expression of the kinematic density (\ref{kdensity2}. Let $(Q; e_1(Q), e_2(Q), T)$ be the moving frame obtained from the standard frame $(O;\mathring{e}_1, \mathring{e}_2, T)$ at the original $O$ by the left-invariant translation to the point $Q$ and the angle $\phi$ that makes $e_1(Q)$ with the standard vector $\mathring{e}_1(Q)$ (as Figure \ref{Fig2}).
\begin{figure}[!ht]
   \includegraphics[width=0.8 \textwidth]{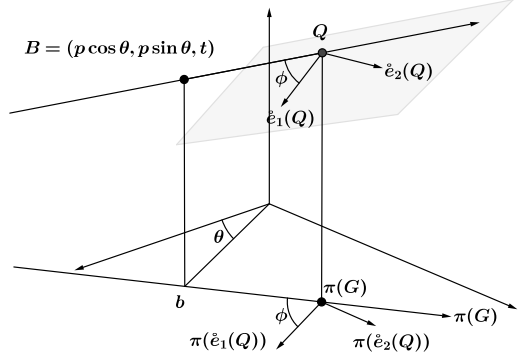}
   \centering \caption{PSH(1)-action}
    \label{Fig2}
\end{figure}
We observe that $e_2(Q)=Je_1(Q) \in \xi_Q$ and the angle $\phi$ indicates the rotation of the frame vector $e_1(Q)$ from the vector $\mathring{e}_1(Q)$ on the contact plane $\xi_Q$. Denote the oriented distance $$h= \pm |\vec{BQ}|,$$ w.r.t. the Levi metric, the sign depends on the direction $\vec{BQ}$ and the orientation of $G$. Since \begin{align*}
Q=(a,b,c)&=B+ h(\sin\theta, -\cos\theta, p) \\
&=(p\cos\theta+h\sin\theta, p\sin\theta-h\cos\theta, t+hp),
\end{align*}we have
\begin{align*}
\left\{
\begin{array}{ll}
a=&p\cos\theta +h\sin\theta, \\
b=&p\sin\theta -h\cos\theta, \\
c=&t+hp, \\
\phi=&\pi/2\pm\theta (\mbox{the sign needs the orientations of line }G).
\end{array}\right.
\end{align*}
Take the derivatives and make the wedge product all together, we reach the other expression (\ref{kdensity2}) of the kinematic density for the group of motions in $H_1$
\begin{align}
da\wedge db\wedge dc \wedge d\phi = dp \wedge d\theta \wedge dt \wedge dh=dG\wedge dh,
\end{align}which is indeed invariant under $PSH(1)$ by $(\ref{density})$.

\section{Proof of Theorem \ref{volume1}}
We shall prove the lemma first. Given a horizontal line $\gamma(s)$ and two points $A=\gamma(s_0)$, $B=\gamma(s_1)$ on $G$. Since $G$ is horizontal, $B\in \xi_A\cap \xi_B$. If we consider the distance joining $A$ and $B$ defined as (\ref{distance}), the infinitesimal length $|\gamma'(s)|$ varies by the Levi-metric defined on different contact plane $\xi_{\gamma(s)}$ for all points $\gamma(s)$ between $A$ and $B$; while by considering $B\in \xi_A$, we define the other distance
\begin{align}\label{distance2}
|AB|_A:=\int^{s_1}_{s_0}|\gamma'(s)|_{\xi_A}ds,
\end{align}where the infinitesimal length $|\gamma'(s)|_{\xi_A}$  depends only on the contact plane $\xi_A$ of the initial point $A$. However, the lemma below shows that both distances are exactly same.

\begin{lem}\label{elength}
Given an oriented horizontal line $\gamma(s)$, parametrized by horizontal arc-length $s$, passing through the points $A=\gamma(s_0)$ to $B=\gamma(s_1)$. Then two distances defined by $(\ref{distance})$ and $(\ref{distance2})$ are same, $|AB|_A=|AB|$.
\end{lem}
\begin{proof}
Note that since $\gamma(s)$ is horizontal, $A\in \xi_B$ and $B\in \xi_B$. Thus, the distance $|AB|_A=|AB|_B$. Therefore, we can always regard any point between between $A$ and $B$ being parametrized as the end of the vector ejecting from $A$ by
$\gamma(s):= A+s(\sin\theta\mathring{e}_1(A)-\cos\theta\mathring{e}_2(A))$. Clearly, its tangent $\gamma'(s)=\sin\theta\mathring{e}_1(A)-\cos\theta\mathring{e}_2(A)$ is a unit vector on $\xi_A\cap\xi_{\gamma(s)}$ for any $s\in [s_0,s_1]$, so $|\gamma'(s)|_{\xi_A}=|\gamma'(s)|_{\xi_{\gamma(s)}}=1$ and the result follows immediately.
\end{proof}

Now we prove Theorem \ref{volume1}.
\begin{proof}
First we observe that the slope of projection $\pi(G)$ of $G$ on $xy$-plane is $-\cot\theta$, which is independent of the orientation of $G$. Now fixed a pair of $(p, \theta)$ and consider the cross-section of domain $D$ and the vertical plane along the projection $\pi(G)$
\begin{align*}
S_{p,\theta}=\Big\{ (x,y,t)\in \mathbb{R}^3; \ p&=x\cos\theta+y\sin\theta, \\
(x,y)&\in \pi(G\cap D), \ t\in I_2 \text{ for some interval }I_2 \Big\}.
\end{align*}
Since the projection $\pi(S_{p,\theta})$ onto $xy$-plane again is $\pi(G)$, we may set the first two coordinates of points on $S_{p,\theta}$ satisfying
$$
y=y_{p,\theta}(x)=p-x\cot\theta.
$$
Thus, for $\theta\neq 0$ or $\pi$, the plane $S_{p,\theta}$ can be parametrized by
$$X:(u,v)\in I_1\times I_2 \mapsto (x(u,v),y(u,v),t(u,v)),$$ for some interval $I_1$, $I_2$ depending the range of domain $D$, where
\begin{align}\label{xyz1}
x(u,v)&=u, \nonumber \\
y(u,v)&=y_{p,\theta}(u)=p-u\cot\theta,\\
t(u,v)&=v. \nonumber
\end{align}
Now we use the following Lemma.
\begin{lem}[\cite{CHL} Lemma 8.7]\label{equivalent}
Let $E=\alpha\mathbb{X}_u + \beta\mathbb{X}_v$ be the tangent vector field defined on the regular surface $X(u,v)$ in $H_1$. Then the vector $E$ is also on the contact bundle $\xi$ (and hence in $TH_1\cap\xi$) if and only if pointwisely the coefficients $\alpha$ and $\beta$ satisfy
\begin{align}\label{belonginboth2}
\alpha(t_u+xy_u-y x_u)+ \beta(t_v+xy_v-y x_v) =0.
\end{align}
\end{lem}
Since $X(u,v)\cap \xi_{X(u,v)}$ is an one-dimensional foliation (a horizontal line in this case) $E$ restricted on $S_{p,\theta}$, $E$ is a linear combination of $X_u$ and $X_v$. By Lemma \ref{equivalent}, we choose $\alpha:=-(t_v+xy_v-yx_v)$ and $\beta:= (t_u+xy_u-yx_u)$ which satisfy $(\ref{belonginboth2})$. Use (\ref{xyz1}), we have
\begin{align}\label{Elength1}
E:= E(u,v)&:=- (t_v+xy_v-yx_v)X_u+(t_u+xy_u-yx_u)X_v  \nonumber \\
&=-(1+x\cdot 0-y\cdot 0)(1,y',0)+(0+xy'-y)(0,0,1)  \nonumber \\
&=(-1,-y',xy'-y) \nonumber \\
&=(-1)\mathring{e}_1(x,y,z)+(-y')\mathring{e}_2(x,y,z).
\end{align}
By Lemma \ref{elength} and (\ref{Elength1}),
\begin{align}
\sigma&=|E|_A=|E| \label{Elength2} \\
&= \int_{q\in E}\sqrt{<E(q),E(q)>_{\xi_q}}\nonumber \\
&=\int_{u\in I_1}\sqrt{1+(y')^2}du \nonumber \\
&=\int_{u\in I_1} |\csc\theta|du. \nonumber
\end{align}
When $\theta=0$ or $\pi$, the horizontal lines lie on the planes $S_{p,\theta}$ parallel to the $yt$-plane, and therefore $d\theta=0$, which implies the density
\begin{align}\label{zeropi}
dG=0.
\end{align}

Finally, combine both cases $(\ref{Elength2})(\ref{zeropi})$,
\begin{align*}
\int_{G_{p,\theta, t}\cap D\neq \emptyset} \sigma dG &= \int_{\theta\neq 0, \pi} \sigma dG+ \int_{\theta=0, \pi } \sigma dG\\
&=\int \Big(\int_{x\in I_1}|\csc\theta|dx  \Big) dt\wedge dp\wedge d\theta\\
&=\int\int_{x\in I_1} |\csc\theta|dx \wedge dt \wedge (dx\cos\theta + dy\sin\theta) \wedge d\theta\\
&=\int\int_{x\in I_1} dx\wedge dt \wedge dy \wedge d\theta\\
&=2\pi V(D),
\end{align*}where we use the fact that $dx\wedge dy \wedge dt$ is Lebesque volume form in $\mathbb{R}^3$ in the last identity, and complete the proof.
\end{proof}

Remark: if $D$ consists of finitely many simply-connected subsets, then the right-hand-side of (\ref{volume1}) becomes the sum of volumes for each subset.

\section{Proof of Theorem \ref{randomm}}
\begin{proof}
Give a fixed line segment $v$ outside of the boundary $\Sigma$. There exists an horizontal line $G$ along the vector $v$ penetrating $D$ at two points. Now move $v$ in the direction of $G$ from \textit{outside} of $\Sigma$ until the end of $v$ leaving $\Sigma$. The condition $v\cap D\neq \emptyset$ gives the distance $\sigma+\ell$ that the end of the vector $v$ traveling over $G$, so by Theorem \ref{volume1} and (\ref{integraldg}) we have
\begin{align*}
m(v; v\cap D\neq \emptyset)
&=\int_{v\cap D\neq \emptyset}dp\wedge d\theta \wedge dt \wedge dh\\
&=\int_{G\cap D\neq \emptyset}(\sigma+ \ell)dG\\
&=2\pi V(D)+2 \l  \mbox{ p-Area}(\Sigma).
\end{align*}
\end{proof}


\begin{thebibliography}{99}
  
  \bibitem{Baddeley} A. Baddeley, \textit{A Fourth Note on Recent Research in Geometrical Probability}, Adv. in Appl. Probability, Vol. 9, No. 4 (1977), pp. 824-860.
  \bibitem{BFG} M. Beals, C. Fefferman, R. Grossman, \textit{Strictly Pseudoconvex Domains in $\mathbb{C^n}$}, Bull. Amer. Math. Soc. (N.S.) V(8), No.2, (1983), 125-322.
  \bibitem{Blaschke} W. Blaschke, \textit{Integralgeometrie 2:Zu Ergebnissen von M. W. Crofton}, Bull. Math. Soc. Roumaine Sci., 37, {1935}, 3-7.

  \bibitem{CCG} O. Calin, D.C. Chang, P. Greiner, \textit{Geometric Analysis on the Heisenberg Group and Its Generalizations}, AMS Studies in Adv. Math., 40, 2007.
 
  \bibitem{CDG} L. Capogna, D. Danielli, N. Garofalo, \textit{An isoperimetric inequality and the geometric Sobolev embedding for vector fields}, Math. Res. Lett. 1 (2), (1994), 203-215.
  \bibitem{CDPT}, L. Capogna, D. Danielli, S.D. Pauls, J. Tyson, \textit{An Introduction to the Heisenberg Group and the Sub-Riemannian Isoperimetric Problem}, Progress in Mathematics, 2007.
  \bibitem{Chern} S.S. Chern, \textit{Lectures on integral geometry}, notes by Hsiao, H.C., Academia Sinica, National Taiwan University and National Tsinghua University, (1965).

  \bibitem{CM} S.S. Chern, J.K. Moser, \textit{Real Hyperserufaces in Complex Manifolds}, Acta Math. 133, (1974), 219-271.


  \bibitem{CHMY1} J.H.Cheng, J.F.Hwang, A.Malchiodi, P.Yang, \textit{Minimal surfaces in pseudohermitian geometry}, Ann. Scuola Norm. Sup. Pisa Cl. Sci. (5), Vol. IV(2005), 129-177.
  \bibitem{CHMY2} J.H. Cheng, J.F. Hwang, A. Malchiodi, P. Yang, \textit{A Codazzi-like equation and the singular set for $C^1$ smooth surfaces in the Heisenberg group}, J. Reine Angew. Math., 671, (2012), 131–198.
  \bibitem{CCHY1} J.H. Cheng, H.L. Chiu, J.F. Hwang, P. Yang, \textit{Umbilic hypersurfaces of constant sigma-k curvature in the Heisenberg group}, Calc. Var. Partial Differential Equations 55 (2016), no. 3, Art.66, 25 pages.
  \bibitem{CHY1} J.H. Cheng, J.F. Hwang, P. Yang, \textit{Regularity of $C^1$ smooth surfaces with prescribed p-mean curvature in the Heisenberg group}, Math. Ann., 344, (2009), no. 1, 1–35.
  \bibitem{CHY2} J.H. Cheng, J.F. Hwang, P. Yang, \textit{Existence and uniqueness for p-area minimizers in the Heisenberg group}, Math. Ann., 337, (2007), no. 2, 253–293.


  \bibitem{CHL} H.L.Chiu, Y.C. Huang, S.H.Lai, \textit{The application of the moving frame method to Integral Geometry in Heisenberg group}, Preprint, http://arxiv.org/abs/1509.00950.
  
  \bibitem{FSC}, B. Franchi, R. Serapoioni, F. S. Cassano, \textit{Rectifiability and perimeter in the Heisenberg group}, Math. Ann. 321 (3), (2001), 479-531.
    
  \bibitem{Howard} R. Howard, \textit{The Kinematic Formula in Riemannian Homogeneous Spaces}, Mem. Amer. Math. Soc. No. 509, V106, 1993.

  \bibitem{LM} G.P. Leonardi, S. Masnou, \textit{On the isoperimetric problem in the Heisenberg group $H^n$}, Annali di Matematica Pura ed Applicata, V184,4, (2005), pp 533-553.
  
  \bibitem{Little} D.V. Little, \textit{A Third Note on Recent Research in Geometrical Probability}, Adv. in Appl. Probability, Vol. 6, No. 1 (1974), pp. 103-130.     
  \bibitem{Malchiodi} A. Malchiodi, \textit{Minimal surfaces in three dimensional pseudo-hermitian manifolds}, Lect. Notes Semin. Interdiscip. Mat., 6, Semin. Interdiscip. Mat. (S.I.M.), Potenza, 2007.
         
  \bibitem{Moran1} P.A.P. Moran, \textit{A Note on Recent Research in Geometrical Probability}, J. of Appl. Probability, Vol. 3, No. 2 (1966), pp. 453-463.
  \bibitem{Moran2} P.A.P. Moran, \textit{A Second Note on Recent Research in Geometrical Probability}, Adv. in Appl. Probability, Vol. 1, No. 1 (1969), pp. 73-89.
  
  \bibitem{RR1} M.Ritor\'e, C. Rosales, \textit{Area-stationary surfaces in the Heisenberg group $H^1$}, Adv. in Math., 219 (2008), pp.633-671.

  \bibitem{RR2} M. Ritor\'e, C. Rosales, \textit{Rotationally invariant hypersurfaces with constant mean curvature in the Heisenberg group $H^n$}, J. Geom. Anal., 16(4), (2006), 703-720.
  
  \bibitem{San} L. A. Santal\'o, \textit{Integral Geometry and Geometric Probability}, Cambridge Mathematical Library, 2nd edition,  {2004}.
  
  \bibitem{SW} R.Schneider, W. Weil, \textit{Stochastic and Integral Geometry}, Springer, 2008.
   \bibitem{Tanaka} N. Tanaka, \textit{A differential geometric study on strongly pseudo-convex manifolds}, Khokuniya, Yokyo, 1975.
  \bibitem{Webster} S. Webster, \textit{Pseudo-hermitian structures of a real hyper-surfaces}, J. Diff. Geo., 13(1978), 25-41.
  
  
  \bibitem{Yang1} P. Yang, \textit{Minimal Surfaces in CR geometry}, Geometric analysis and PDEs, 253–273, Lecture Notes in Math., 1977, Springer, 2009,
  \bibitem{Yang2} P. Yang, \textit{Pseudo-Hermitian Geometry in 3-D}, Milan J. Math. V(79), (2011), 181-191.
  
  \bibitem{Zhang} G.Y. Zhang, \textit{A lecture on Integral Geometry}, Proceedings of the fourteenth Internatioanl Workshop on Diff. Geom. 14(2010), 13-30.
  \bibitem{Zhou1} J.Z. Zhou, \textit{Kinematic formulas in Riemannian spaces}, Proceedings of the tenth Internatioanl Workshop on Diff. Geom. 10(2006), 39-55.
  \bibitem{Zhou2} J.Z. Zhou, \textit{Kinematic Formulas for Mean Curvature Powers of Hypersurfaces and Hadwiger's Theorem in $R^{2n}$}, Trans. Amer. Math. Soci., v345, No1, 1994, pp. 243-262.
  \bibitem{ZL} J.Z. Zhou, M. Li, \textit{Kinematic Formulas of Total Mean Curvatures for Hypersurfaces}, Chinese Ann. Math. 37B(1), 2016, 137–148.
  
  
  
  

 
 

 
  
 
  
\end{thebibliography}
\end{document}